\documentclass[12pt]{amsart}

\usepackage{amsfonts,amssymb,stmaryrd,amscd,amsmath,latexsym,amsbsy,color, hyperref, changes}

\usepackage{amssymb}
\usepackage{amsfonts}
\usepackage{latexsym}

\usepackage[notcite,notref,final]{showkeys}

\newtheorem{theorem}{Theorem}[section]
\newtheorem{lemma}[theorem]{Lemma}
\newtheorem{proposition}[theorem]{Proposition}
\newtheorem{corollary}[theorem]{Corollary}
\theoremstyle{definition}
\newtheorem{definition}[theorem]{Definition}
\newtheorem{notation}[theorem]{Notation}

\newtheorem{example}[theorem]{Example}

\newtheorem{remark}[theorem]{Remark}

\newcommand{\Aut}{\text{Aut}}

\newcommand{\g}{\mathfrak{g}}

\newcommand{\ben}{\begin{enumerate}}
\newcommand{\een}{\end{enumerate}}

\theoremstyle{plain}

\newtheorem*{sol}{Solution}

\theoremstyle{definition}

\theoremstyle{remark}

\newcommand{\solu}[1]{\begin{sol}{\bf (\ref{#1})}}

\begin{document}

\title[Finite dimensional Hopf actions on alg. quantizations]{Finite dimensional Hopf actions on algebraic quantizations}

\author{Pavel Etingof}
\address{Department of Mathematics, Massachusetts Institute of Technology,
Cambridge, MA 02139, USA}
\email{etingof@math.mit.edu}

\author{Chelsea Walton}
\address{Department of Mathematics, Temple University,
Philadelphia, PA 19122, USA}
\email{notlaw@temple.edu}


\maketitle


\begin{abstract} Let $k$ be an algebraically closed field of characteristic zero. 
In joint work with J. Cuadra \cite{CEW1,CEW2}, we showed that a semisimple Hopf action on a Weyl algebra over a polynomial algebra 
$k[z_1,\dots,z_s]$ factors through a group action, and this in fact holds for any finite dimensional Hopf action if $s=0$. 
We also generalized these results to finite dimensional Hopf actions on algebras of differential operators. 
In this work we establish similar results for Hopf actions on other algebraic quantizations of  
commutative domains. This includes universal enveloping algebras of finite dimensional Lie algebras, spherical symplectic reflection algebras, quantum Hamiltonian reductions of Weyl algebras (in particular, quantized quiver varieties), finite W-algebras and their central reductions, quantum polynomial algebras, twisted homogeneous coordinate rings of abelian varieties, and Sklyanin algebras. The generalization in the last three cases uses a result from algebraic number theory, due to A. Perucca.  
\end{abstract}
\section{Introduction} 

Throughout the paper, $k$ will denote an algebraically closed field of
characteristic zero.
In \cite[Theorem~1.3]{EW1}, we showed that any semisimple 
Hopf action on a commutative domain over $k$
factors through a group action. 
Likewise, it was established in our joint work with Juan Cuadra that
the same conclusion holds for semisimple Hopf actions on Weyl algebras
${\bf A}_n(k[z_1,\dots,z_s])$ \cite[Proposition~4.3]{CEW1}. Moreover, we
showed that it holds for any (not necessarily semisimple) finite
dimensional Hopf action on ${\bf A}_n(k)$ \cite[Theorem~1.1]{CEW2},
and, more generally, on algebras of differential operators of smooth
affine varieties \cite[Theorem~1.2]{CEW2}. Finally, in \cite{EW3} we 
extended these results to certain finite dimensional Hopf actions on deformation quantizations
(i.e., formal quantum deformations) of commutative domains. We say that there is {\it No Finite Quantum Symmetry} in the settings above.

The goal of this paper is establish No Finite Quantum Symmetry results for  finite dimensional Hopf actions on other
algebraic quantizations of commutative domains, i.e., quantizations whose parameters are elements of $k$ 
(rather than formal variables). We now summarize our main results for various classes of algebraic 
quantizations. 

\subsection{Semisimple Hopf actions on filtered quantizations} 
Our first main result concerns Hopf actions on {\it filtered deformations} (or {\it filtered quantizations}) of 
commutative domains, that is, on filtered $k$-algebras $B$ where  the associated graded algebra gr($B$) is  a commutative 
finitely generated domain. 

Let  $B$ be a $\Bbb Z_+$-filtered algebra over $k$ 
such that ${\rm gr}(B)$ is a commutative finitely generated 
domain. We will see that  for sufficiently large primes $p$, 
the algebra $B$ admits a reduction $B_p$ modulo $p$, which is a
domain over $\overline{\Bbb F}_p$. Namely, there exists an $R$-order $B_R\subset B$
over some finitely generated subring $R\subset k$, and $B_p=B_{\psi,p}:=B_R\otimes_R\overline{\Bbb F}_p$ for a homomorphism $\psi: R\to \overline{\Bbb F}_p$.
(For details on $R$-orders in $B$ see Section~2.1 below). 

Recall  that a ring  $A$ is {\it PI} if it satisfies a polynomial identity over $\mathbb{Z}$.
By Posner's and Ore's theorems \cite{Pos,Ore} \cite[Theorem~13.6.5 and Corollary~1.14]{MR},  a domain $A$ is PI if and only if it is an Ore domain and 
its division ring of fractions  ${\rm Frac}(A)$ is a central division algebra. In this case,  ${\rm Frac}(A)$ is a division ring that is dimension $d^2$ over its center, where $d$ is the {\it PI degree} of $A$ \cite[Definition~13.6.7]{MR}.

\begin{definition} Given $B$ as above, we say that $B$  is an {\it algebra with PI reductions}
if it admits an order $B_R$ such that $B_p$ is PI for sufficiently large $p$ (with any choice of $\psi$).
\footnote{It follows from Lemma \ref{exis}(ii) below that if this condition holds for one pair $(R,B_R)$, then the condition  holds for all such pairs.}
\end{definition}

\begin{theorem}
[Theorem~\ref{main1}] 
\label{main1intro} 
If $B$ is an {\it algebra with PI reductions}, then 
any semisimple Hopf action on $B$ 
factors through a group action. 
\end{theorem}

Note that when the Hopf action preserves the filtration of $B$, Theorem \ref{main1intro} (even without the PI reduction
 assumption) 
is proved in \cite[Proposition~5.4]{EW1}; our main achievement here is that we eliminate this requirement.  

A basic example of an algebra with PI reductions 
is the Weyl algebra \linebreak $B = {\bf A}_n(k)$, and, in fact, the proof of 
Theorem~\ref{main1intro} is analogous to the proof of
\cite[Theorem~4.1]{CEW1} which addresses this case. 
Moreover, a wide range of filtered quantizations 
(each defined in Subsection~\ref{examp} below) 
are algebras with PI reductions, resulting in the following corollary. 

\begin{corollary}[Corollary~\ref{modp-examples1}] \label{modp-examples}
Let $B$ be one of the following filtered $k$-algebras:
\begin{enumerate}
\item[(i)] Any filtered quantization $B$ generated in filtered degree one; 
in particular, the enveloping algebra $U({\mathfrak{g}})$ of a 
finite dimensional Lie algebra $\mathfrak{g}$, or 
the algebra $D_\omega(X)$ of twisted differential operators on a smooth 
affine irreducible variety $X$; 
\item[(ii)] a finite W-algebra or its quotient by a central character; 
\item[(iii)] a quantum Hamiltonian reduction of a Weyl algebra 
by a reductive group action; in particular, the coordinate ring of a quantized quiver variety;
\item[(iv)] a spherical symplectic reflection algebra; or
\item[(v)] the tensor product of any of the algebras above with any commutative finitely generated 
domain over $k$.
\end{enumerate}
Then, any semisimple Hopf action on $B$ factors through a group action.
\end{corollary} 

Other applications of Theorem~\ref{main1intro} have been investigated recently by Lomp and Pansera   \cite{LP}; for instance, they establish No Finite Semisimple Quantum Symmetry on certain iterated differential operator rings.

\begin{remark} We do not  know if a filtered quantization of a finitely generated commutative domain over $k$ must be an algebra with PI reductions (i.e., if the PI reduction assumption is, in fact, vacuous), see the question in \cite[Introduction]{CEW1} and \cite[Question~1.1]{Et2}. This is of independent interest in noncommutative ring theory.
\end{remark}

 \subsection{Finite dimensional Hopf actions on filtered quantizations} 
 
 Similarly to \cite[Theorem~4.2]{CEW1}, Theorem~\ref{main1intro} and hence Corollary~\ref{modp-examples} hold for Hopf-Galois actions of any (not necessarily semisimple) finite dimensional Hopf algebra  (Theorem~\ref{main1a}). The proof is 
parallel to the proofs of Theorem \ref{main1intro} and  \cite[Theorem~4.2]{CEW1}. 

Moreover, it turns out that even without the Hopf-Galois assumption, Theorem~\ref{main1intro} extends to non-semisimple
Hopf actions for a somewhat more restrictive class of quantizations. To see this, let us recall some algebras introduced in \cite{CEW2}.

\begin{notation}[$B$, $B_{p^m}$, $C_m$, $D_{p^m}$, $Z$, $Z(m)$] \label{notation-B}
Let  $B$ be a quantization with PI reductions, and let $B_{p^m}$ be the reduction of $B$ modulo $p^m$. 
Let $C_m$ be the center of $B_{p^m}$, ${\rm Frac}(C_m)$ be its ring of fractions,  
and $D_{p^m}:=B_{p^m}\otimes_{C_m}{\rm Frac}(C_m)$. The  PI reduction condition implies 
that $D_{p^m}$ is  the full localization (i.e., ring of fractions) of $B_{p^m}$. 
These  algebras are defined over the truncated Witt ring $W_{m,p}$ of $\overline{\Bbb F}_p$ 
(cf. \cite[Subsections~2.1,~2.3,~2.4]{CEW2}). Let $Z$ be the center of the central 
 division algebra $D_p$. (Here and below, to lighten the notation, we will often suppress dependence on $p$.) 
Let $Z_m$ be the center of $D_{p^m}$, and let
$Z(m)$ be its image in $D_p$ under the map $D_{p^m} \twoheadrightarrow D_p$ (so $Z(1)=Z$). It is easy to see that $Z(m)\subset Z(m-1)$. 
\end{notation}

\begin{definition}\label{nondege} 
We say that  an algebra $B$ with PI reductions is {\it nondegenerate} if  for almost all $p$ one has 
$\textstyle \bigcap_{m\ge 1}Z(m)=\overline{\Bbb F}_p$. 
\end{definition}

\begin{theorem}[Theorem \ref{anyact1}] \label{anyact}If $B$ is a nondegenerate algebra with PI reductions,
then any finite dimensional Hopf action on $B$ factors through a group action (i.e., the condition that $H$ is semisimple in 
Theorem~\ref{main1intro} can be dropped). 
\end{theorem} 

The proof of Theorem \ref{anyact} is similar to the proof of \cite[Theorem~1.2]{CEW2}. 

To illustrate when the nondegeneracy condition holds, recall  that ${\rm gr}(B)$ carries a natural Poisson bracket. 
Namely, if $B$ is commutative, this bracket is zero; 
otherwise, if $d$ is the largest integer such that 
$[F_iB,F_jB]\subset F_{i+j-d}B$, then 
for $a_0$ $\in {\rm gr}_i(B)$ and $b_0\in {\rm gr}_j(B)$, 
the Poisson bracket 
$\lbrace a_0,b_0\rbrace$ is the projection of $[a,b]$ to ${\rm gr}_{i+j-d}(B)$, 
where $a\in F_iB$, $b\in F_jB$ are any lifts of $a_0,b_0$ respectively. 
Thus, \linebreak ${\rm gr}(B)=O(X)$, where $X$ is an irreducible Poisson algebraic 
variety. 

The nondegeneracy assumption is satisfied, in particular, when $X$ is a {\it generically symplectic} Poisson variety, i.e., one having a symplectic dense open subset; see Theorem~\ref{nondegen}.  Therefore, Theorem \ref{anyact} holds for many of the examples of Corollary~\ref{modp-examples}
--- quantum Hamiltonian reductions of Weyl algebras, central reductions of finite W-algebras, spherical symplectic reflection algebras, and tensor products thereof (see Corollary \ref{gensymp}).  

\subsection{Quantum polynomial algebras}
For our next  main result, we consider finite dimensional Hopf actions on {\it quantum polynomial algebras} (or {\it quantized coordinate rings of  affine $n$-space}):
$$k_{\bold q}[x_1,\dots, x_n] := k \langle x_1,\dots, x_n \rangle/ (x_ix_j-q_{ij}x_jx_i),$$
where ${\bold q}=(q_{ij})$, $q_{ij}\in k^\times$ with $q_{ii}=1$ and $q_{ij}q_{ji}=1$. Thus we can view ${\bold q}$ as a point of  
the algebraic torus $(k^\times)^{n(n-1)/2}$ with coordinates $q_{ij}$ for $i<j$. 

There are many examples of semisimple Hopf actions on $k_{\bold q}[x_1,\dots, x_n]$ that do not factor through 
group actions;  the parameters $q_{ij}$ are roots of unity in these examples. See, for instance, \cite[Theorem~0.4]{CKWZ}, \cite[Example~5.10]{EW1}, and \cite[Examples~7.4--7.6]{KKZ}. 
Still, we establish the following result.
  
Let $\langle {\bold q}\rangle $ be the subgroup in $(k^\times)^{n(n-1)/2}$ 
generated by $\bold q$, and let $G_{\bold q}$ be its Zariski closure. 
Let $G_{\bold q}^0$ be the connected component of the identity in $G_{\bold q}$. 

\begin{theorem}[Theorem~\ref{main3}] \label{main3intro} Let $H$ be a semisimple Hopf algebra of dimension $d$. If the order of $G_{\bold q}/G_{\bold q}^0$ is coprime to $d!$, then any $H$-action on $B:=k_{\bold q}[x_1,\dots, x_n]$ factors through a group action. 
\end{theorem}

If each $q_{ij}$ is a root of unity of order  $r_{ij}$,  then  $|G_{\bold q}/G_{\bold q}^0| = \text{lcm}\{r_{ij}\}_{i<j}$. In particular,
if $n=2$, i.e. if $B=k\langle x,y\rangle/(xy-qyx)$, 
then the condition on $\bold q=q\in k^\times$ in Theorem \ref{main3intro} is that the order of $q$ is coprime to $d!$ or infinite. On the other hand, the condition on $\bold q$ in Theorem~\ref{main3intro} is also satisfied if each $q_{ij}$ is a non-root of unity and the set of the $q_{ij}$ is multiplicatively independent; here, $|G_{\bold q}/G_{\bold q}^0| = 1$. See Example~\ref{exaaa} for a discussion of how to compute $|G_{\bold q}/G_{\bold q}^0|$ in general. 

One may compare Theorem~\ref{main3intro} to a similar result, \cite[Theorem~4.3]{CWZ}, in the case where the Hopf action preserves the grading of $k_{\bold q}[x_1,\dots, x_n]$. 
But note that without the degree-preserving assumption, semisimplicity is still needed in Theorem~\ref{main3intro}; see \cite{EW2}, \cite[Example~3.6]{EW3} for counterexamples for $n=1, 3$, respectively.

Moreover,  Theorem~\ref{main3intro} is valid for finite dimensional Hopf algebras in the Hopf-Galois case, where we can replace the condition 
``coprime to $d!$" with ``coprime to $d$" (Proposition~\ref{HopfGal}). Also, Theorem~\ref{main3intro} generalizes straightforwardly (with the same proof) to 
actions on the quantum torus $k_{\bold q}[x_1^{\pm 1},\dots,x_n^{\pm 1}]$.

Another generalization of Theorem~\ref{main3intro} to the non-semisimple case can be made under a nondegeneracy assumption. Recall that $\bold q$ may be viewed as a skew-symmetric bicharacter 
on $\Bbb Z^n$ with values in $k^\times$, with $\bold q(e_i,e_j)=q_{ij}$ 
for the standard basis $\lbrace e_i\rbrace$. A   bicharacter 
$\bold q$ is called {\it nondegenerate} if the character $\bold q(a,?): \Bbb Z^n\to k^\times$ is non-trivial whenever $a\ne 0$. Note that unlike skew-symmetric bilinear forms (which are always degenerate in odd dimensions), a skew-symmetric bicharacter can be 
nondegenerate for any $n\ge 2$.

\begin{theorem}[Theorem \ref{main4}]\label{main4intro} Let $H$ be a finite dimensional Hopf algebra of dimension $d$ acting on $B:=k_{\bold q}[x_1,\dots,x_n]$. Assume that the order of $G_{\bold q}/G_{\bold q}^0$ is coprime to $d!$, and $\bold q$ is a nondegenerate bicharacter. Then, the action of $H$ on $B$ factors through a group algebra. 
\end{theorem}

It is shown in Example~\ref{Swee} that Theorems~\ref{main3intro} and~\ref{main4intro} fail when hypotheses are removed; these examples  involve actions of the non-semisimple 4-dimensional Sweedler Hopf algebra.

\subsection{Twisted homogeneous coordinate rings of abelian varieties and Sklyanin algebras}\label{ellip}
Let $X$ be an abelian variety over $k$, ${\mathcal L}$ be an ample line bundle on $X$, and let $\sigma: X\to X$ be an automorphism given by translation by a point $s \in X$. Then we can define the twisted homogeneous coordinate ring 
$$
B(X,\sigma,{\mathcal L}):=\textstyle \bigoplus_{n=0}^\infty H^0(X,\otimes_{i=0}^{n-1} (\sigma^i)^*{\mathcal{L}}), 
$$ 
with twisted multiplication $f*g:=f(\sigma^n)^*(g)$, where $f$ is of degree $n$ (\cite{AV}). 
It is well-known  that $B(X,\sigma,{\mathcal{L}})$ is a domain, and if $|\sigma|<\infty$, then $B(X,\sigma,{\mathcal{L}})$ is a PI domain of PI degree $|\sigma|$. 

Let $G_\sigma$ be the Zariski closure of the subgroup $\{s^i\}_{i \in \mathbb{Z}}$, and let  $G_\sigma^0$ be the connected component of the identity in $G_\sigma$. 

\begin{theorem}[Theorem~\ref{main5}] \label{main3aintro} If $H$ is a semisimple Hopf algebra of dimension $d$, 
and if the order of $G_\sigma/G_\sigma^0$ 
is coprime to $d!$, then any $H$-action on $B(X,\sigma,{\mathcal{L}})$ 
factors through a group action. 
\end{theorem}

In particular, if the subgroup $\{s^i\}_{i \in \mathbb{Z}}$  is Zariski-dense in $X$, then any semisimple Hopf action on $B(X,\sigma,{\mathcal{L}})$ 
factors through a group action. Moreover,  if $X=:E$ is an elliptic curve, the condition on $\sigma$ in Theorem~\ref{main3aintro} is that the order of $\sigma$ is coprime to $d!$ or infinite. 

Lastly, we study semisimple Hopf actions on another class of quantizations: the {\it 3-dimensional Sklyanin algebras} $S(a,b,c)$ (Definition~\ref{def:Skly3thcr}).
To $S(a,b,c)$, one can naturally associate an elliptic curve $E_{abc}\subset \mathbb{P}^2_k$ and an automorphism $\sigma_{abc}$ given by translation by a point. (See \cite[Introduction]{ATV}.)

\begin{theorem} \label{mainSklyintro}
If $H$ is a semisimple Hopf algebra of dimension $d$, and if the order of  $\sigma_{abc}$ 
is coprime to $d!$ or infinite, then any $H$-action on $S(a,b,c)$ 
factors through a group action. 
\end{theorem}

\begin{remark}
We believe that by adapting the techniques in this work, one could establish a version of Theorem~\ref{mainSklyintro} for semisimple Hopf actions on other elliptic algebras, such as in \cite{Skl} (or, see \cite{SS}) and in \cite{EG}, \cite{Ode}, \cite{Ste}. 
Further, we believe that under an appropriate nondegeneracy condition, there are No Finite Quantum Symmetry results for such elliptic algebras and for twisted homogeneous coordinate rings $B(X,\sigma,\mathcal{L})$ as above; compare to Theorem~\ref{main4intro}.
\end{remark}

Our paper is organized as follows. We discuss semisimple Hopf actions on filtered quantizations in Section 2, non-semisimple Hopf actions on filtered quantizations in Section 3, 
semisimple Hopf actions on quantum polynomial algebras in Section 4, non-semisimple Hopf actions on quantum polynomial algebras in Section 5, and Hopf actions on twisted homogeneous coordinate rings of abelian varieties and Sklyanin algebras in Section 6. 
 The results of Sections 4-6 rely on a number-theoretic result of Antonella Perucca discussed in the Appendix (Section 7). 
 \smallskip
 
\begin{center} 
The notation and terminology of the introduction is used\\ throughout this work, often without mention.
\end{center}

\section{Semisimple Hopf actions on filtered quantizations} \label{sec:filtered}

\subsection{The result on semisimple Hopf actions on quantizations  with PI reductions}

In this section, we study actions of semisimple Hopf algebras $H$ on filtered quantizations $B$.
Throughout this section, we let $B$ denote a $\Bbb Z_+$-filtered algebra over $k$ such that ${\rm gr}(B)$ 
is a commutative finitely generated domain; such $B$ will be referred to as a filtered quantization. 

Our goal is to prove Theorem~\ref{main1intro} (Theorem~\ref{main1} below). This result was established in \cite{CEW1} for $B$ being a Weyl algebra and we generalize the techniques of \cite{CEW1} for our purpose here.

Let $R$ be a finitely generated subring of $k$. By an {\it $R$-order} in a filtered quantization $B$ 
we  mean an $R$-subalgebra $B_R$ of $B$ such that ${\rm gr}(B_R)$ 
is a finitely generated $R$-algebra which is projective as an 
$R$-module, and so that the natural map ${\rm gr}(B_R)\otimes_R k\to {\rm gr}(B)$ is an isomorphism of graded $k$-algebras.  

\begin{lemma}\label{exis} \textnormal{(i)} Any filtered quantization $B$ admits an $R$-order $B_R$ for a suitable ring $R$. 

\textnormal{(ii)} For any two orders $B_{R}$ over $R$ and $B_{R'}$ 
over $R'$, there exists a finitely generated ring $R''\subset k$ containing $R,R'$ 
and an $R''$-algebra isomorphism $\phi: B_R\otimes_{R}R''\to B_{R'}\otimes_{R'}R''$ 
such that ${\rm gr}\phi$ is an isomorphism. 
\end{lemma} 

\begin{proof} (i) By the Hilbert basis theorem, the algebra ${\rm gr}(B)$ is finitely presented. This implies that so is $B$, as 
we can lift the generators and defining relations of ${\rm gr}(B)$ to those of $B$.  

More specifically, pick homogeneous generators $\bar b_1,\dots,\bar b_n$ of ${\rm gr}(B)$
of degrees $m_1,\dots,m_n$. Let 
$$
p_s(\bar b_1,\dots,\bar b_n)=0,\ 
s=1,\dots,r
$$ 
be a set of defining relations for ${\rm gr}(B)$, 
with $p_s\in k[X_1,\dots,X_n]$ being homogeneous of degree $d_s$ (this set may be chosen to be  finite by the Hilbert basis theorem). 
Let $b_j$ be lifts of $\bar b_j$ to $B$, and $\widetilde{p}_s$ be homogeneous lifts 
of $p_s$ to $k\langle X_1,\dots,X_n\rangle$. Then $[b_i,b_j]=f_{ij}(b_1,\dots,b_n)$, where 
$f_{ij}\in k\langle X_1,\dots,X_n\rangle$ is a noncommutative polynomial  
of filtration degree at most $m_i+m_j-1$, and 
$\widetilde{p}_s(b_1,\dots,b_n)=p_s'(b_1,\dots,b_n)$, where 
$p_s'\in k\langle X_1,\dots,X_n\rangle$ is a noncommutative polynomial  of filtration degree at most 
$d_s-1$. 

Let $g_s:=\widetilde{p}_s-p_s'\in k\langle X_1,\dots,X_n\rangle$. Thus, we have relations
\begin{equation}\label{rela}
[b_i,b_j]=f_{ij}(b_1,\dots,b_n) ~~\text{ and }~~ g_s(b_1,\dots,b_n)=0
\end{equation}
in $B$. It is easy to see that these relations are defining, since they are already defining at the graded level. 

Using relations \eqref{rela}, we can find a suitable finitely generated subring 
$R\subset k$ and define $B_R\subset B$ as follows. We take $\widetilde{R}$ to be the ring generated 
by all the coefficients of the polynomials $f_{ij}$, $g_s$, and set $R=\widetilde{R}[1/f]$ for a suitable $f\in \widetilde{R}$. Now let $B_R$ be the subalgebra of $B$ generated over $R$ by $b_1,\dots,b_n$. We  can choose $f$ so that \eqref{rela} are defining relations for $B_R$, and so that $B_R$ is an $R$-order on $B$,  since  for a suitable choice of $f$,
${\rm gr}(B_R)$ is a projective (in fact, free) $R$-module  by Grothendieck's Generic Freeness Lemma (\cite[Theorem~14.4]{Eis}).  
 This proves (i). 
 
\smallskip

(ii) Note that we have a natural isomorphism of filtered algebras 
\linebreak $\widetilde{\phi}: B_R\otimes_R k\to B_{R'}\otimes_{R'} k$ (as both are equal to $B$). 
This isomorphism is defined over some finitely generated ring $R''\subset k$ containing $R,R'$, 
which implies (ii). 
\end{proof} 

\begin{lemma} \label{lem:filtered} Suppose that $B$ is a filtered quantization that 
carries an action of a finite dimensional Hopf algebra $H$. Let $S$ be a finitely generated subring of $k$, and 
$B_S$ be an $S$-order in $B$. Then, one can find a finitely generated subring $R\subset k$ containing $S$ 
and a Hopf order $H_R\subset H$ (cf. \cite[Definition 2.1]{CEW1}) so that there is an induced action of $H_R$ on $B_R:=B_S\otimes_S R$
which gives the action of $H$ on $B$ upon tensoring over $R$ with $k$.
\end{lemma}

\begin{proof} We use the method of \cite[Section~2]{CEW1}. 
Pick homogeneous generators $\bar b_1,\dots,\bar b_n$ of ${\rm gr}(B_S)$,
and let $b_j$ be lifts of $\bar b_j$ to $B_S$.
Choose a basis $\lbrace{h_m\rbrace}$ of $H$. We have 
\begin{equation}\label{actfor}
h_m \cdot b_j =q_{mj}(b_1,\dots,b_n)
\end{equation}
for some noncommutative polynomials $q_{mj}\in k\langle X_1,\dots,X_n\rangle$. 
Let $R$ be generated over $S$ by the structure constants of $H$ in the basis $\lbrace h_m\rbrace$ 
and the coefficients of $q_{mj}$. Let $H_R\subset H$ be the span of $h_m$ over $R$.  
Then, $H_R\subset H$ is a Hopf order, and $H_R$ acts on $B_R$
by formula \eqref{actfor}. 
The lemma is proved.   
\end{proof}

Thus, any action of $H$ on $B$ admits an $R$-order for some finitely generated ring $R\subset k$. Moreover, it is easy to see from Lemma \ref{exis}(ii) that any two such 
orders over rings $R,R'$ can be identified after tensoring with some finitely generated 
ring $R''\subset k$ containing $R,R'$, so an order is essentially unique. 

Now fix a ring $R$ and an $R$-order $B_R\subset B$ with an action of $H_R$. 
Let $p$ be a sufficiently large prime, and $\psi: R\to \overline{\Bbb F}_p$ be a character. 
Following \cite[Section~2]{CEW1},  let $H_p := H_R\otimes_R \overline{\Bbb F}_p$, $B_p := B_R\otimes_R \overline{\Bbb F}_p$  be the corresponding reductions of $H,B$ modulo $p$,
where $\overline{\Bbb F}_p$ is an $R$-module via $\psi$. Then, $H_p$ acts on $B_p$ (by applying $\psi$ to the action of $H_R$ on $B_R$). 

\begin{lemma}  \label{lem:EGA} For a sufficiently large prime $p$, ${\rm gr}(B_p)$, and hence $B_p$, is a domain. 
\end{lemma}

\begin{proof} We have ${\rm gr}(B_p)={\rm gr}(B_R)\otimes_R \overline{\Bbb F}_p$. 
Therefore, the statement follows from \cite[9.7.7(i)]{EGA} 
(``geometric irreducibility'').
\end{proof}

\begin{theorem}\label{main1} 
If $B$ is  an algebra with PI reductions, then 
any semisimple Hopf action on $B$ 
factors through a group action. 
\end{theorem}

\begin{proof} We may assume without loss of generality that the action of $H$ on $B$ is inner faithful 
(otherwise we can pass to an action of a quotient Hopf algebra). 

Take $p \gg 0$.  Then by \cite[Proposition~2.4]{CEW1} (which applies with the same proof in 
our more general situation), $H_p$ acts inner faithfully on $B_p$. 
Moreover, as in \cite[Lemma~2.5]{CEW1},  $H_p$ is a semisimple
cosemisimple Hopf algebra over $\overline{{\Bbb F}}_p$.

Since $B$ is  an algebra with PI reductions, by Lemma~\ref{lem:EGA}, the algebra 
$B_p$ is a PI domain. Let $D_p$ be the division algebra of quotients of $B_p$.  
Then by \cite[Corollary~3.2(ii)]{Et2},  $D_p$ is a
central division algebra of degree $p^n$ for some $n\ge 0$ (which may depend on $p$). 
Moreover, $H_p$ acts inner faithfully on $D_p$ by \cite[Theorem~2.2]{SV}.

 Since $\deg D_p=p^n$ is coprime to $(\dim H)!$,  \cite[Proposition~3.3(ii)]{CEW1} implies that  $H_p$ is cocommutative. 
Thus, $H$ is cocommutative  (as in the proof of \cite[Theorem~4.1]{CEW1}), and therefore is a group algebra. 
\end{proof} 

\subsection{Some examples of filtered quantizations}\label{examp} 

As a consequence, Theorem~\ref{main1} applies to semisimple Hopf actions on many classes of filtered quantizations. Namely, we will consider the following examples.

\subsubsection{Twisted differential operators.} Let $X$ be a smooth affine irreducible algebraic variety over $k$, and $\omega$ a closed 2-form on $X$. Then we define the {\it algebra of twisted 
differential operators $D_\omega(X)$} to be the algebra generated by $O(X)$ and 
elements $L_v$ attached $k$-linearly to vector fields $v\in {\rm Der}O(X)$ on $X$, 
with defining relations
$$
L_{fv}=fL_v,\quad [L_v,f]=v(f),\quad [L_v,L_w]=L_{[v,w]}+\omega(v,w)
$$
 for $f\in O(X),\ v,w\in {\rm Der}O(X)$. 
Then, $D_\omega(X)$ carries a filtration defined by $\deg O(X)=0$ and $\deg L_v=1$ for $v\in 
{\rm Der}O(X)$, and ${\rm gr}(D_\omega(X))=O(T^*X)$, the algebra of functions on the symplectic variety $T^*X$. 

The filtered algebra $D_\omega(X)$ depends only on the cohomology class $[\omega]$ of $\omega$, 
and if $[\omega] = 0$,  then $D_\omega(X)=D(X)$, the algebra of usual differential operators on 
$X$. For more on twisted differential operators, see e.g. \cite[Section~2]{BB}.  

\subsubsection{Quantum Hamiltonian reductions.} 
Let $G$ be a reductive algebraic group over $k$
with Lie algebra $\g$, and let $(V, (\cdot, \cdot))$  be a faithful finite dimensional symplectic representation of $G$. Let ${\bold A}(V)$
be the Weyl algebra of $V$, generated by $v\in V$ with relations $[v,w]=(v,w)$ for 
$v,w\in V$. We have a natural action of $G$ on ${\bold A}(V)$ which preserves its filtration. In this case, we have a natural $G$-equivariant Lie algebra map $\widehat\mu: \g\to {\bold A}(V)$ 
called the {\it quantum moment map}, which quantizes the classical moment map 
$\mu: V\to \g^*$ where $\mu(v)(a)=\frac{1}{2}(v,av)$ for $v\in V$, $a\in \g$.
Now, given a character $\chi: \g\to k$,  we can define the algebra $$B(\chi):=[{\bold A}(V)/{\bold A}(V)(\widehat{\mu}(a)-\chi(a), ~a\in \g)]^G,$$ called the {\it quantum Hamiltonian reduction of ${\bold A}(V)$ by $G$ using $\chi$}. It inherits a filtration from the Weyl algebra.  See \cite[Chapter~4]{Et1} for further details. 

Assume that  the moment map 
$\mu$ is flat, and that the scheme $\mu^{-1}(0)$ is reduced and irreducible (i.e., 
$\mu^{-1}(0)$ is a reduced irreducible complete intersection). 
In this case, $X:=\mu^{-1}(0)/G$ is an irreducible generically symplectic Poisson 
 variety, and $B(\chi)$ is a filtered quantization of $O(X)$. 
 See \cite[Section~2.3]{Lo2} and \cite[part~2(i)]{Nak}, and references therein for more details. 
 
 An interesting special case of this is when $G=(\prod_{i=1}^r GL(V_i))/k^\times$, and 
$V=\oplus_{i,j} (V_i\otimes V_j^*)^{\oplus m_{ij}}$, where 
$V_i$ are finite dimensional vector spaces and $m_{ij}=m_{ji}$ are 
positive integers where $m_{ii}$ is even; i.e., $V$ is the space of representations of a doubled quiver, and $G$ is the group of linear transformations for this quiver. 
In this case, $B(\chi)$ is the {\it quantized quiver variety}, see e.g. \cite[Section~3.4]{BPW}.
The conditions under which $\mu$ is flat and $\mu^{-1}(0)$ is reduced and irreducible 
are given in \cite[Theorems~1.1 and~1.2]{CB}. 
 
\subsubsection{Finite W-algebras.} Let $\g$ be a simple Lie algebra over $k$, and $e\in \g$ be a nilpotent element. To this data one can attach a Lie subalgebra ${\mathfrak{m}}\subset \g$
with a character $\chi$, and a {\it finite W-algebra} is 
$$
U(\g,e):=
(U(\g)/U(\g)(a-\chi(a), ~a\in {\mathfrak{m}}))^{\textrm{ad } \mathfrak{m}},
$$
a quantum Hamiltonian reduction of $U(\g)$. The algebra $U(\g,e)$ has 
a filtration induced by the filtration in $U(\g)$, and its associated graded algebra is a polynomial algebra (of functions on the corresponding Slodowy slice).  We refer 
the reader to \cite[Subsections~2.3 and~2.4]{Lo1} and the references therein for details.

Also, the center $U(\g)^\g$ of $U(\g)$ embeds into $U(\g,e)$, 
so for any central character $\theta: U(\g)^\g\to k$ 
one can consider the central reduction $U_\theta(\g,e):=U(\g,e)/(a-\theta(a),~a\in U(\g)^\g)$. 
Then ${\rm gr}(U_\theta(\g,e))=O(X)$, where $X$ is the nilpotent Slodowy slice, a 
generically symplectic Poisson variety. 

\subsubsection{Symplectic reflection algebras.}  Let $G$ be a finite group and $V$ a faithful finite dimensional symplectic representation of $G$, and assume that $V$ is not a direct sum of two nonzero symplectic representations. The {\it symplectic reflection algebra} 
$H_{t,c}(G,V)$ is the most general filtered deformation of $kG\ltimes SV$, where $[F_i,F_j]\subset F_{i+j-2}$; here $t\in k$ and $c$ is a conjugation invariant function on the set of symplectic reflections in $G$ (see \cite[Chapter~8]{Et1}). 

Let $\bold e=|G|^{-1}\sum_{g\in G}g$
be the symmetrizing idempotent for $G$. Then, the algebra $\bold e H_{t,c}(G,V)\bold e$ is called the {\it spherical symplectic reflection algebra}. For 
$t=1$, it is a filtered quantization of 
$O(X)$ where $X=V/G$, a generically symplectic Poisson variety. 

\begin{remark}
There are many other interesting examples of filtered quantizations, and 
our results given below can be extended to such examples. 
Since this extension is rather routine, we leave it outside the scope of this paper.  
\end{remark} 

\subsection{Results on semisimple Hopf actions on specific 
filtered quantizations} 
Here are some concrete applications of Theorem~\ref{main1}. 

\begin{corollary}\label{modp-examples1}
Let $B$ be one of the following filtered $k$-algebras:
\begin{enumerate}
\item[(i)] Any filtered quantization $B$ generated in filtered degree one; 
in particular, the enveloping algebra $U({\mathfrak{g}})$ of a 
finite dimensional Lie algebra $\mathfrak{g}$, or 
the algebra $D_\omega(X)$ of twisted differential operators on a smooth 
affine irreducible variety $X$; 
\item[(ii)] a finite W-algebra or its quotient by a central character; 
\item[(iii)] a quantum Hamiltonian reduction of a Weyl algebra 
by a reductive group action; in particular, the coordinate ring of a quantized quiver variety;
\item[(iv)] a spherical symplectic reflection algebra; or
\item[(v)] the tensor product of any of the algebras above with any commutative finitely generated 
domain over $k$.
\end{enumerate}
Then, any semisimple Hopf action on $B$ factors through a group action.
\end{corollary} 

Note that in some of these cases, a stronger statement is true: any finite dimensional 
(not necessarily semisimple) Hopf action on $B$ factors through a group action, see Corollary \ref{gensymp} 
below. However, we still prefer to prove the weaker version here, since the proof is simpler 
(e.g., it does not require reduction modulo prime powers).  

\begin{proof}
By Theorem \ref{main1}, our job is to show that  $B$ is  an algebra with PI reductions.  In other words, we need to show that the division algebra $D_p$ of fractions of $B_p$ is central (i.e., there is a ``$p$-center") for $p \gg 0$. We do so below in each case.  

\smallskip

(i) We will show that if a filtered quantization $A$ of a commutative finitely generated domain 
$A_0$ over a field $F$ of characteristic $p>0$ is generated in degree one, then it is 
module-finite over its center after localization; this implies the required statement. 

Let $A_0[i]$ be the degree $i$ part of $A_0$. Then $A_0[0]=A[0]$ is a finitely generated commutative domain. Let $\bar a_1,\dots,\bar a_n$ be generators of $A_0[1]$ 
as an $A_0[0]$-module. Let $a_i$ be lifts of $\bar a_i$ to $A$. 
Then, $a_i$ and $A[0]$ generate $A$ as an algebra. Also the operators $[a_i,?]$ 
are derivations of $A[0]$, hence vanish on $A[0]^p$. Thus, 
$A[0]^p$ is central in $A$. Let $K$ be the field of quotients of 
$A[0]^p$, and $A':=A\otimes_{A[0]^p}K$. The $K$-algebra $A'$ is generated in filtration degree $1$, and $L:=F_1A'$ is a finite dimensional vector space over $K$ 
(as it is spanned by $1,a_1,\dots,a_n$ over $A[0]\otimes_{A[0]^p}K$, and $A[0]$ is module-finite over $A[0]^p$ as $A[0]$ is a finitely generated algebra). Also, $L$ is closed under commutator. Thus, $L$ is a finite dimensional Lie algebra over $K$, and 
$A'$ is a quotient of the enveloping algebra $U(L)$. But the enveloping algebra 
of a finite dimensional Lie algebra in characteristic $p$ is module-finite over its center (i.e., there is a $p$-center; 
see \cite{Ja1}, \cite[Chapter 6, Lemma 5]{Ja2}). 
This implies that $A'$ is module-finite over its center, as desired. 

\smallskip

(ii) Since a $W$-algebra is a quantum Hamiltonian reduction of the enveloping algebra 
$U(\mathfrak{g})$  of a semisimple Lie algebra ${\mathfrak{g}}$, (see \cite{Lo1}), the  statement  follows from (i).  

\smallskip

(iii) This also follows from (i) and the definition of the quantum Hamiltonian reduction.

\smallskip

(iv) This holds by \cite[Theorem~9.1.1 (in the appendix)]{BFG}.

\smallskip

(v) This follows easily from the previous cases. 
\end{proof}

\section{Finite dimensional Hopf actions on filtered quantizations} 

\subsection{Hopf-Galois actions} 

Theorem \ref{main1} does not hold for nonsemisimple Hopf actions, as there are many inner faithful actions of nonsemisimple finite dimensional Hopf algebras on commutative domains, see \cite{EW2}. However, Theorem~\ref{main1} is valid in the Hopf-Galois case.

\begin{theorem}\label{main1a} 
Let $B$ be a filtered quantization of a commutative finitely generated domain with PI reductions,  and let $H$ be a finite dimensional Hopf algebra over $k$ which acts on $B$. Assume that this action gives rise to an $H^*$-Hopf-Galois extension $B^H\subset B$. Then, 
$H$ is a group algebra. 
\end{theorem}

\begin{proof}
The result follows from the arguments in the proofs of Theorem~\ref{main1}
and \cite[Theorem~4.2]{CEW1}.  Namely, recall Notation \ref{notation-B}. We show similarly to the proof of Theorem \ref{main1} that $Z$ is $H_p$-stable, 
and then proceed as in the proof of \cite[Theorem~4.2]{CEW1}. 
Specifically, by \cite[Corollary 3.2(ii))]{Et2}, $D_p$ has degree $p^n$ over its center $Z=Z(D_p)$ for some $n$, so 
by \cite[Proposition 3.3(i)]{CEW1}, $Z$ is $H_p$-invariant. Now, 
since the action of $H$ on $B$ gives rise to a Hopf-Galois extension, so does the action of $H_p$ on $Z$, i.e., the algebra map 
$Z\otimes_{Z^{H_p}}Z\to Z\otimes H_p^*$ is an isomorphism. 
Thus, $H_p^*$ is commutative and $H_p$ is cocommutative, so $H$ is cocommutative (\cite[Lemma 2.3(ii)]{CEW1}), i.e., a group algebra by the Cartier-Gabriel-Kostant theorem 
\cite[Corollary~5.6.4(3) and Theorem~5.6.5]{Mo}.   
\end{proof}

\subsection{Preparatory results on nondegenerate quantizations} 
Another generalization of Theorem \ref{main1} concerns nondegenerate quantizations, defined in Definition \ref{nondege}. To obtain it, we will first 
need to generalize \cite[Theorem~3.2]{CEW2}.
Let $\mathcal{H}$ be a finite dimensional Hopf algebra over an algebraically closed field $F$ of characteristic $p>0$, and let
$\mathcal{Z}$ be a finitely generated field extension of $F$.
Let $\mathcal{Z}(m)$, for $m\ge 1$, be a collection of subfields of $\mathcal{Z}$ such that $\mathcal{Z}(m)\supset \mathcal{Z}(m+1)$ for all $m\ge 1$. 

\begin{theorem}\label{th3.2gen} Suppose that $\bigcap_{m\ge 1}\mathcal{Z}(m)=F$, and that $[\mathcal{Z}:\mathcal{Z}(m)]$ is a power of $p$ for all $m\ge 1$. Assume that $\mathcal{H}$ acts $F$-linearly and inner faithfully on $\mathcal{Z}$.
If $p>\dim \mathcal{H}$ and $\mathcal{H}$ preserves $\mathcal{Z}(m)$ for all $m$, then $\mathcal{H}$ is a group algebra.  
\end{theorem} 

\begin{proof} The proof is the same as that of \cite[Theorem~3.2]{CEW2}. Indeed, the only properties 
of the fields $Z^{p^m}$ used in that  proof are that their intersection is $F$ and 
that the degree of $Z$ over $Z^{p^m}$ is a power of $p$.  
\end{proof} 

We will also need the lemma below from commutative algebra. 
Introduce  the following notation. Let 
$W_N=W_N(F):=W(F)/(p^N)$ be the $N$-th truncated Witt ring of $F$ ($W_N$ is an algebra over 
$\Bbb Z/p^N\Bbb Z$; cf. \cite[Subsection 2.1]{CEW2}). Let $Y$ be an irreducible smooth 
affine algebraic variety over $F$ with structure algebra $A:=O(Y)$, 
and $\widetilde{Y}$ be a flat deformation of $Y$ over $W_N$. 
Let $1\le m\le N$, and let $A_m:=O(\widetilde{Y})/(p^m)$ (a free $\Bbb Z/p^m\Bbb Z$-module); thus $A_1=A$ and $A_{m-1}=A_m/(p^{m-1})$ for $m\ge 2$. Let 
$$d_m: A_m\to \Omega_{A_m/W_m}$$ be the 
differential. 

\begin{lemma} \label{commalg} For $1\le m\le N$, the image of 
${\rm ker}(d_m)$ in $A$ is $A^{p^m}$.
\end{lemma} 

\begin{proof} It is clear that the image of 
${\rm ker}(d_m)$ contains $A^{p^m}$,  so 
it remains to establish the opposite inclusion.
We will do so by induction in $m$. 

The base of induction is the equality 
${\rm ker}(d|_A)=A^p$, which is 
the Cartier isomorphism in degree zero
(\cite[Section~7]{Kat}). Alternatively, here is a
direct proof. Since $A$ is integrally closed in its quotient field \linebreak $L:={\rm Frac}(A)$, we may replace $A$ with $L$. Note that 
$L$ can be represented as a finite 
separable extension of $F(y_1,\dots,y_n)$, where \linebreak $n=\dim Y$.  
Given $f\in L$ such 
that $df=0$, consider the minimal polynomial $P(t)=t^r+a_{r-1}t^{r-1}+\cdots+a_0$
of $f$ over $E:=F(y_1,\dots,y_n)$. 
Applying the differential to  the equation $P(f)=0$, we get 
$\sum_{j=0}^{r-1}f^jda_j=0$. Since $P$ is the minimal polynomial, this implies that $da_j=0$ for all $j$. Thus $a_j\in E^p$ (as the statement in question 
is easy for purely transcendental fields). 
Thus, $E^p(f)$ is a finite separable extension of $E^p$   
(as $P$ is a separable polynomial). But $E^p(f)$ is a purely inseparable extension of 
$E^p(f^p)$. Hence, $E^p(f)=E^p(f^p)$, that is, 
$f\in E^p(f^p)\subset L^p$, as desired. 

To perform the induction step, 
suppose $f\in {\rm ker}(d_m)$. 
Our job is to show that the image $\bar f$ 
of $f$ in $A$ is contained in $A^{p^m}$. 
By the induction assumption we know 
that $\bar f=b^{p^{m-1}}$ for some $b\in A$, 
so it remains to show that $b=c^p$ for some $c\in A$. 

For this, let us expand $f$ in a 
power series in some local coordinate system $y_1,\dots,y_n$ 
on $\widetilde{Y}$. It is easy to see by looking at monomials 
that if $g\in W_m[[y_1,\dots,y_n]]$ and $dg=0$ then
the reduction $\bar g$ of $g$ modulo $p$ lies in 
$F[[y_1^{p^m},\dots,y_n^{p^m}]]$.  In particular, the power series 
expansion of $\bar f$ in $y_i$ is in $F[[y_1^{p^m},\dots,y_n^{p^m}]]$.
This means that the power series expansion of $b$ is in 
$F[[y_1^p,\dots,y_n^p]]$. Thus, $db=0$. By the base of induction we conclude that 
$b=c^p$ for some $c\in A$, which completes the induction step. 
\end{proof}

Moreover, we will need the result below.

\begin{lemma} \label{lem:fieldfinite}
Let $B$ be an algebra with PI reductions, and let $D_p$ denote the full localization (i.e., the ring of fractions) 
of the reduction $B_p$ of $B$, for $p \gg 0$. Then,  the center $Z$ of $D_p$ is a finitely generated field extension of $\overline{\Bbb F}_p$.
\end{lemma}

\begin{proof}
Let $v_1, \dots, v_N$ be a basis of $D_p$ over $Z$, and  let $b_1,\dots,b_n$ be generators of $B_p$. Then, $b_sv_i=\sum_{j=1}^N \beta_{si}^j v_j$, for $\beta_{si}^j\in Z$. Let $K$ denote the field $\overline{\Bbb F}_p(\beta_{si}^j)$.

Now take $z\in Z$. Then $z\in D_p$, so $z=c^{-1}b$, and hence $cz=b$, for some $b,c\in B_p$ with $c\ne 0$. Since $b,c\in B_p$, they are noncommutative polynomials in $b_1,\dots,b_n$ over $\overline{\Bbb F}_p$. So, $bv_i=\sum \beta_i^jv_j$, $cv_i=\sum \gamma_i^jv_j$, with $\beta_i^j,\gamma_i^j\in K$. But $\gamma_i^jz=\beta_i^j$ and $\gamma_i^j$ are not all zero. So,
$z\in K$ and hence $Z=K$. Thus, $Z$ is a finitely generated extension of $\overline{\Bbb F}_p$.
\end{proof}

\subsection{Hopf actions on nondegenerate quantizations} 

Now let $B$ be a filtered quantization with PI reductions. 

\begin{theorem}\label{anyact1}  If $B$ is a nondegenerate  algebra  with PI reductions,
then any finite dimensional Hopf action on $B$ factors through a group action (i.e., the condition that $H$ is semisimple in 
Theorem~\ref{main1} can be dropped). 
\end{theorem} 

\begin{proof}
The proof is obtained by combining the proofs of Theorem \ref{main1} and \cite[Theorem~1.1]{CEW2}  with the following modifications. 
\smallskip

1. In \cite[Lemma~2.5]{CEW2} and below, $x_i,y_i$ should be replaced by any finite set of generators $L_1,\dots, L_r$ of $B$, 
and the number $2n$ in the proof of \cite[Lemma~4.3]{CEW2} should be replaced by $r$
(cf. \cite[proof of Theorem~1.2]{CEW2}). 
\smallskip

2. The discussion in \cite[Subsection~2.4, Lemma~4.7, Proposition~4.8]{CEW2} (needed to justify the assumptions of \cite[Theorem~3.2]{CEW2})
becomes unnecessary. Instead, note that if $a\in D_{p^m}$ is central modulo $p^{m-1}$ for some $m\ge 2$, then $a^p$ is central. Hence $Z(m)\supset Z(m-1)^p$, implying that $Z(m)\supset 
Z^{p^{m-1}}$ and therefore $[Z:Z(m)]$ is finite (by Lemma~\ref{lem:fieldfinite}) and is a power of $p$. Now the proof proceeds by invoking Theorem \ref{th3.2gen},
whose assumptions are satisfied by the nondegeneracy property of $B$
and using a straightforward generalization of \cite[Lemma~4.6]{CEW2}.  
\smallskip

3. The rest of the proof of \cite[Theorem~1.1]{CEW2}
is modified as in the proof of Theorem \ref{main1}. Namely, we use 
the PI reduction condition and \cite[Corollary~3.2]{Et2} 
which says that the PI degree of $B_p$ is a power of $p$.  
\end{proof}

The next theorem shows that the nondegeneracy assumption is satisfied, in particular, when 
${\rm gr}(B)=O(X)$, where $X$ is generically symplectic. 
 
\begin{theorem}\label{nondegen} Let $B$ be a  quantization  with PI reductions, and assume that 
${\rm gr}(B)=O(X)$, where $X$ is a generically symplectic Poisson variety. Then, 
any action of a finite dimensional Hopf algebra $H$ on $B$ factors through a group action. 
\end{theorem} 

\begin{proof} 
By Theorem \ref{anyact1}, it suffices to show that $B$ is a nondegenerate quantization, i.e., that $\bigcap_{m\ge 1}Z(m)=\overline{\Bbb F}_p$, for $p \gg 0$. 

Recall  Notation~\ref{notation-B}.
Let $C$ be the center of $B_p$; thus, by Posner's theorem \cite[Theorem~13.6.5]{MR}, the field ${\rm Frac}(C)$ of fractions of $C$ is $Z$. 
Let $C_m$ be the center of $B_{p^m}$, and $C(m)$ be its image in $B_p$. 

Let $a\in B_{p^m}$ be central modulo $p$ (i.e., 
the image $\overline{a}$ of $a$ in $B_p$ lies in $C$). Then,  
$a^p$ is central modulo $p^2$, $a^{p^2}$ is central modulo $p^3$, and so on. Hence,$C^{p^{m-1}}\subset C(m)$. Let $C_m'$ be the preimage of $C^{p^{m-1}}$ in $C_m$. 
Then the image of $C_m'$ in $B_p$ is $C^{p^{m-1}}$. 

We claim that 
\begin{equation}\label{nondeg1}
Z(m) = \text{Frac}(C(m)).
\end{equation}
Indeed, it is clear that Frac($C(m)$)$\subseteq Z(m)$. On the other hand,
 observe that any element $a\in D_{p^m}$ can be written as $a=c^{-1}b$, where 
$c\in C_m'$ is nonzero modulo $p$, 
and $b\in B_{p^m}$ (as this can be done modulo $p$, since 
$D_p=Z^{p^{m-1}}B_p$). Now given $z\in Z(m)$, let $\widetilde{z}$ be its lift to $Z_m$. 
Writing $\widetilde{z}=c^{-1}b$ as above, we see that $b:=c\widetilde{z}\in C_m$. 
Let $\overline{b}\in C(m)$, $\overline{c}\in C^{p^{m-1}}\subset C(m)$ 
be the reductions of $b,c$ modulo $p$, respectively. 
We have $\overline{b}=\overline{c}z$, hence $z=\overline{c}^{-1}\overline{b}\in {\rm Frac}
(C(m))$, as claimed.  

Now let $B_{0p^m}:={\rm gr}(B_{p^m})$. This is a Poisson algebra over the truncated Witt ring $W_{m,p}$. Let $C_{0m}$ be the Poisson center of $B_{0p^m}$, and $C_0(m)$ be the image of $C_{0m}$ in $B_{0p}$. Then ${\rm gr}(C_m)\subset C_{0m}$ and hence 
\begin{equation} \label{nondeg2}
\textrm{gr}(C(m))
\subset C_0(m). 
\end{equation}

Let $Z_0:={\rm Frac}(B_{0p})$. Since $X$ is generically symplectic, 
$C_{0m}$ coincides with the set of all $f\in B_{0p^m}$ such that $df=0$ . By Lemma~\ref{commalg} (taking $\widetilde{Y}$ to be the reduction modulo $p^m$  of a symplectic dense affine open subset $U\subset X$), this implies that 
\begin{equation}\label{nondeg3}
\text{Frac}(C_0(m))\subset Z_0^{p^m}.
\end{equation}

Now suppose that $z\in \bigcap_{m\ge 1}Z(m)$ with $z\ne 0$. Then by \eqref{nondeg1}, for each $m$,
we have $z=f_m/g_m$ for $f_m,g_m\in C(m)$. 
Let $f_m^0, g_m^0\in {\rm gr}(C(m))$ 
be the leading terms of $f_m,g_m$. By \eqref{nondeg2}, $f_m^0, g_m^0\in C_0(m)$. Then for any $m,n$ we have 
$f_m^0g_n^0=f_n^0g_m^0$ since $f_mg_n=f_ng_m$. 
So $z_0:=f_m^0/g_m^0$ is independent of $m$ and by \eqref{nondeg3} belongs  
to $Z_0^{p^m}$ for all $m\ge 0$. Since $\bigcap_{m\ge 1}Z_0^{p^m}$ is a perfect field that is finitely generated over $\overline{\Bbb F}_p$, we get that $\bigcap_{m\ge 1}Z_0^{p^m} = \overline{\Bbb F}_p$. So, $z_0\in \overline{\Bbb F}_p$ is a nonzero 
constant, and $f_m^0=z_0g_m^0$ for all $m$; in particular, 
\begin{equation}\label{eqdeg}
\deg(f_m)=\deg(g_m).
\end{equation}  
Now $z-z_0=\frac{f_m-z_0g_m}{g_m}$, and the numerator has degree strictly less than $\deg{g_m}$. 
This violates \eqref{eqdeg}, so $z-z_0=0$, i.e. $z\in \overline{\Bbb F}_p$.  
This proves the theorem.  
\end{proof} 

\begin{corollary} \label{gensymp} 
Let $B$ be one of the following algebras: 

\begin{enumerate}
\item[(i)] a quotient of a finite W-algebra by a central character; 

\item[(ii)] a Hamiltonian reduction of a Weyl algebra
by a reductive group action; in particular, the coordinate ring of a quantized quiver variety;

\item[(iii)] a spherical symplectic reflection algebra $H_{1,c}(G,V)$; or

\item[(iv)] the tensor product of any of the algebras in \textnormal{(i)-(iii)}.
\end{enumerate}

\noindent Then, any action of a finite dimensional Hopf algebra $H$ on $B$ factors through a group action. 
\end{corollary} 

\begin{proof} 
It is explained in Subsection \ref{examp} that in examples (i)-(iv), we have that ${\rm gr}(B)=O(X)$ where $X$ is generically symplectic. This  implies the corollary.   
\end{proof} 

\begin{proposition}\label{main1b} 
Theorems \ref{main1},~\ref{main1a},~\ref{anyact1},~\ref{nondegen} remain valid if $B$ is replaced by its 
quotient division algebra ${\rm Frac}(B)$.  
\end{proposition}

\begin{proof} The proofs are obtained by combining the proofs of 
Theorem \ref{main1}, \ref{main1a}, \ref{anyact1}, \ref{nondegen} with the proof of \cite[Proposition~4.4]{CEW1}. (The exact form of the generators of $B$ used in the proof of
\cite[Proposition~4.4]{CEW1} is irrelevant for the argument.)   
\end{proof}

\section{Semisimple Hopf actions on quantum polynomial algebras}  \label{sec:qpoly}

We now turn to finite dimensional Hopf actions on quantum polynomial algebras
$$k_{\bold q}[x_1,\dots,x_n] := k \langle x_1, \dots x_n \rangle/ (x_ix_j-q_{ij}x_jx_i),$$
where ${\bold q}=(q_{ij})$, $q_{ij}\in k^\times$ with $q_{ii}=1$ and $q_{ij}q_{ji}=1$. We 
view ${\bold q}$ as a point of the algebraic torus $(k^\times)^{n(n-1)/2}$ with coordinates $q_{ij}$, $i<j$. Let $\langle {\bold q}\rangle $ be the subgroup in $(k^\times)^{n(n-1)/2}$ 
generated by $\bold q$, and let $G_{\bold q}$ be its Zariski closure. 
Let $G_{\bold q}^0$ be the connected component of the identity in $G_{\bold q}$. 

\begin{theorem} \label{main3} Let $H$ be a semisimple Hopf algebra of dimension $d$. If the order of $G_{\bold q}/G_{\bold q}^0$ is coprime to $d!$, then any $H$-action on \linebreak$B:=k_{\bold q}[x_1,\dots,x_n]$ factors through a group action. 
\end{theorem}

\begin{proof}
We may assume that $H$ acts on $B:=k_{\bold q}[x_1,\dots,x_n]$ inner faithfully. 
Let $R\subset k$ be a finitely generated subring containing $q_{ij}$, let \linebreak
$B_R:=R_{\bold q}[x_1,\dots,x_n]$ be the quantum polynomial algebra defined over $R$, 
and $H_R$ be a Hopf $R$-order with an action on $B_R$ which becomes 
the action of $H$ on $B$ upon tensoring with $k$. 

Similarly to the proof of Theorem~\ref{main1}, we need to control the PI degree of $B_R$ after reducing modulo $p$; we employ a version of the number-theoretic result of A. Perucca (as presented in Appendix) to do so.

Given a number field $K$ and a ring homomorphism $\xi: R\to K$, let $R':=\xi(R)$, 
$H_{R'}:=H_R\otimes_R R'$, $B_{R'}:=B_R\otimes_R R'=R'_{\xi(\bold q)}[x_1,\dots,x_n]$. 
Then, $H_{R'}$ acts on $B_{R'}$ inner faithfully. For a generic choice of $\xi$, any multiplicative relation satisfied 
by $\xi(q_{ij})$ is already satisfied by $q_{ij}$, so by Example~\ref{exaaa}, we have
$|G_{\bold q}/G_{\bold q}^0|=|G_{\xi(\bold q)}/G_{\xi(\bold q)}^0|$. 
By Corollary~\ref{ntcor}, there exist infinitely many primes $p$ with prime 
ideals ${\mathfrak{p}}\subset R'$ lying over them such that, for a generic homomorphism $\psi: R'\to \overline{\Bbb F}_p$
annihilating ${\mathfrak{p}}$, the order $N:=N_{\mathfrak{p}}$ of $\psi\circ \xi(\bold q)$ is finite and 
  relatively prime to $d!$. Let $H_p:=H_{R'}\otimes_{R'}\overline{\Bbb F}_p$ and  $B_p:=B_{R'}\otimes_{R'}\overline{\Bbb F}_p$ be the corresponding reductions of $H$ and $B$ modulo $p$. 
For large enough $p$, the Hopf algebra $H_p$ is semisimple and cosemisimple (by \cite[Lemma~2.5]{CEW1}), and $B_p$ is a PI domain with PI degree dividing $N^n$ (as $x_i^N$ are central elements in $B_p$). Moreover, 
$H_p$ acts on $B_p$ inner faithfully by a version of \cite[Proposition~2.4]{CEW1} adapted to the algebra $B$ (with the same proof).

Let $D_p$ be the quotient division algebra of $B_p$.
Then, the PI degree of $D_p$ divides $N^n$, and  is therefore coprime to $d!$. Further,  $H_p$ acts inner faithfully on $D_p$.   
Hence, \cite[Proposition~3.3(ii)]{CEW1} implies that $H_p$ is cocommutative. 
Since this happens for infinitely many primes, we conclude that $H_{R'}$ is cocommutative.
Since this happens for generic maps $\xi$, this implies that $H_R$ is cocommutative. 
Thus $H$ is cocommutative, i.e., $H$ is a group algebra.    
\end{proof} 

\begin{corollary}\label{main3a} The conclusion of 
Theorem \ref{main3} holds when $q_{ij}=q^{m_{ij}}$, 
where $m_{ij}=-m_{ji}$ are integers, and the order of $q\in k^\times$
is infinite or is coprime to $d!$.  
\end{corollary} 

\begin{proof} This is a special case of Theorem \ref{main3}.
\end{proof} 

\begin{example} The assumption in Theorem \ref{main3} and Corollary \ref{main3a}
that the order of $G_{\bold q}/G_{\bold q}^0$ is coprime to $d!$ 
cannot be removed. For instance, there exists an inner faithful action of the 8-dimensional noncommutative noncocommutative semisimple Hopf algebra on the quantum polynomial algebra $k_{-1}[x,y]$, see   \cite[Example~7.6]{KKZ}. In this case, $|G_{\bold q}/G_{\bold q}^0|=2$. 
\end{example}

\section{Finite dimensional Hopf actions on quantum polynomial algebras.} 

Let us now extend the results of the previous section to 
not necessarily semisimple Hopf algebras, under some additional assumptions. 

First of all, when the action of $H$ on $B$ is Hopf-Galois, we can 
remove in Theorem \ref{main3} the assumption that 
$H$ is semisimple, and also weaken the coprimeness assumption, 
replacing $d!$ with $d$. 

\begin{proposition} \label{HopfGal} 
Suppose that a finite dimensional Hopf algebra $H$ acts on $B:=k_{\bf q}[x_1, \dots, x_n]$, 
and the order of $G_{\bold q}/G_{\bold q}^0$ is coprime to $d$. 
If this action gives rise to an $H^*$-Hopf-Galois extension $B^H\subset B$, then
the action of $H$ on $B$ factors through a group algebra. 
\end{proposition}

\begin{proof}
The proof is parallel to the proof of Theorem \ref{main1a}. 
The weaker coprimeness assumption suffices since 
by the Hopf-Galois condition, $[D_p: D_p^H]=d$ 
(not just $\le d$).  Here, $p \gg 0$ and $D_p$ is the full localization of $B$ reduced modulo $p$  via the method in the proof of Theorem~\ref{main3}.
\end{proof}

Let us now give a generalization of Theorem \ref{main3} to the non-semisimple case under a nondegeneracy assumption. 

\begin{theorem}\label{main4} Let $H$ be a finite dimensional Hopf algebra of dimension $d$ acting on $B:=k_{\bold q}[x_1,\dots,x_n]$. Assume that the order of $G_{\bold q}/G_{\bold q}^0$ is coprime to $d!$, and $\bold q$ is nondegenerate. Then, the action of $H$ on $B$ factors through a group action. 
\end{theorem}

\begin{proof}  
The proof is obtained by combining the proofs of Theorem \ref{anyact1} and Theorem \ref{main3}. 
Let us describe the necessary changes. 

We argue as in the proof of Theorem \ref{main3}. Fix a generic character 
$\xi: R\to K$ from $R$ to a number field $K$, and set $R'=\xi(R)$.
By Corollary \ref{ntcor}, there exist infinitely many primes $p$ with prime 
ideals ${\mathfrak{p}}\subset R'$ lying over them such that, for a generic homomorphism
\linebreak $\psi: R'\to \overline{\Bbb F}_p$
annihilating ${\mathfrak{p}}$, the order 
$N:=N_{\mathfrak{p}}$ of $\psi\circ \xi(\bold q)$ is finite and 
coprime to $d!$. 

Consider the image $Z(m)$ of the center $Z_m$ of $D_{p^m}$ in $D_p$ (thus, $Z(1)=Z$). By a straightforward generalization of 
\cite[Lemma~4.6]{CEW2}, $Z(m)$ is preserved by the action of $H_p$. 
It is clear that $Z(m)$ is generated by the monomials
$x_1^{m_1}\cdots x_n^{m_n}$ such that $\prod_j q_{ij}^{m_j}=1$ in the truncated ring of Witt 
vectors $W_{m,p}$ (see \cite[Subsection~2.1]{CEW2}). 
Let $W_{m,p}'$ be the kernel of the natural map of multiplicative groups 
$W_{m,p}^\times\to \overline{\Bbb F}_p^\times$. 
Then, every element of $W_{m,p}'$ has order a power of 
$p$. Hence, $[Z:Z(m)]$ is a power of $p$. 
Also it is clear from the nondegeneracy condition for $\bold q$  
that $\bigcap_m Z(m)=\overline{\Bbb F}_p$. Thus, Theorem \ref{th3.2gen} 
applies, and yields that $H_p$ is cocommutative. 
Hence $H_{R'}$ is cocommutative, implying that $H_R$ is cocommutative
and ultimately that $H$ is cocommutative, i.e., a group algebra. 
\end{proof} 

\begin{remark} \label{rem:nondeg}
If $q_{ij}=q^{m_{ij}}$ where $q$ is not a root of unity, then $\bold q$ is nondegenerate if and only if $\det(m_{ij})\ne 0$. Theorem \ref{main4} applies in this case.  
This gives a generalization of \cite[Theorem~0.4]{CWZ} to non-homogeneous 
Hopf actions for even $n$.
\end{remark}

\begin{proposition}\label{quo}
Theorem \ref{main3}, Corollary \ref{main3a} and Theorem \ref{main4} remain valid if the quantum polynomial algebra $B$ is replaced by the quantum torus $k_{\bf q}[x_1^{\pm 1}, \dots x_n^{\pm 1}]$ or the division algebra of quotients ${\rm Frac}(B)$.
\end{proposition} 

\begin{proof} In the case of the quantum torus, the proof is analogous to the proof of 
Theorem \ref{main3}. The case of the division algebra of quotients is obtained
using the same argument as in the proof of \cite[Proposition~4.4]{CEW1}. 
\end{proof} 

\begin{example}\label{Swee} The condition that $H$ is semisimple cannot be dropped in Theorem \ref{main3}, and the condition that $\bold q$ is nondegenerate cannot be dropped in Theorem \ref{main4}.  

Namely, let $A=A_0\oplus A_1$ be a $\Bbb Z/2\Bbb Z$-graded domain 
with a nonzero central element $z\in A_1$, and take $H$ to be the $4$-dimensional Sweedler Hopf algebra generated by a group-like element $g$ and a $(g,1)$-skew-primitive element $u$ with $g^2=1$, $u^2=0$ and $gu+ug =0$.

\begin{enumerate}
\item Then, there is an action of $H$ on $A$ (not preserving the grading of $A$) given by $g\cdot a=(-1)^{{\rm deg}a}a$, and 
$u\cdot a=0$ if $a\in A_0$ and $u\cdot a=za$ if $a\in A_1$. It is easy to check that this action is well-defined; it is inner faithful 
since $u$ acts by a nonzero operator. 

\smallskip

\item In particular, we have an inner faithful action of $H$ 
on the quantum polynomial algebra $k_q[x,y]$, for $q$ a root of unity of 
any odd order $2m-1$, $m>0$; namely, we can take $z=x^{2m-1}$. 
\smallskip

\item Also, this gives an inner faithful action of $H$ on 
the quantum torus $k_q[x_1^{\pm1},\dots,x_n^{\pm 1}]$
if $n$ is odd: we can take the central element 
$$
z=x_1x_2^{-1}x_3\cdots x_{n-1}^{-1}x_n.
$$ 
For even $n$, such an action is impossible if $q$ is not a root of unity 
by Theorem \ref{main4}. Indeed, the matrix $m_{ij}:={\rm sign}(j-i)$ is 
nondegenerate if and only if $n$ is even (see Remark~\ref{rem:nondeg}).
\smallskip

\item Finally, this gives an inner faithful Sweedler Hopf algebra action 
on the Weyl algebra $\bold A_n(F)$ when ${\rm char}(F)=p\ge 3$;
the $\Bbb Z/2\Bbb Z$ grading is defined by giving the generators degree $1$, and 
we can take, for instance, $z=x_1^p$.
(Note that by  \cite[Theorem~1.1]{CEW2}, this is impossible in characteristic zero; indeed, the center of $\bold A_n(k)$ is $k$.)
\end{enumerate}
\end{example}

\section{Semisimple Hopf actions on twisted homogeneous coordinate rings and 3-dimensional Sklyanin algebras} \label{sec:Skly}

Now let us consider semisimple Hopf actions on twisted homogeneous coordinate rings of
abelian varieties. 
We keep the notation of Subsection \ref{ellip}. 

Let $H$ be a Hopf algebra over $k$ of dimension $d$. 

\begin{theorem} \label{main5} We have the following statements. 
\begin{enumerate}
\item[(i)] If $H$ is semisimple, and if the order of $G_\sigma/G_\sigma^0$ 
is coprime to $d!$, then any $H$-action on $B:=B(X,\sigma,{\mathcal{L}})$ 
factors through a group action.  
\item[(ii)]  Moreover, part \textnormal{(i)}   holds for the $H$-action on the division algebra of quotients ${\rm Frac}(B)$ of $B$. 
\item[(iii)] Part \textnormal{(i)} also holds for not necessarily semisimple $H$ 
if the order of $G_\sigma/G_\sigma^0$ is coprime to $d$ and the $H$-action gives rise to a Hopf-Galois extension. 
\end{enumerate}
\end{theorem}

\begin{proof}
The proofs of the statements \textnormal{(i),(ii),(iii)}  are parallel to the proofs of Theorem~\ref{main3}, Proposition~\ref{quo}, and Proposition~\ref{HopfGal}, respectively, where we use that the PI degree of $B$ equals the order of $\sigma$. The only difference is that Corollary \ref{ntcor} is applied to the abelian variety $X$ with subgroup $\{s^i\}_{i \in \mathbb{Z}}$ rather than the torus $(k^\times)^{n(n-1)/2}$ with subgroup $\langle {\bf q} \rangle$. 
\end{proof} 

In particular, if $X=:E$ is an elliptic curve, Theorem \ref{main5} holds if the order of $\sigma$ is coprime to $d!$ or infinite. Moreover, if $\sigma$ has infinite order, the assumption that $H$ 
is semisimple can be dropped. 

\begin{theorem}\label{ellcurve} Let $E$  be an elliptic curve, and take $\sigma \in \Aut(E)$ given by translation by a point
of infinite order. Then, any finite dimensional Hopf action on $B(E,\sigma,{\mathcal {L}})$ factors through a group action. 
\end{theorem} 

\begin{proof}
The proof repeats the proofs of Theorem \ref{main4} 
and Theorem \ref{anyact1} without significant changes. 
\end{proof} 

Finally, let us consider semisimple Hopf actions on three-dimensional Sklyanin algebras \cite{ATV, FO}. 
Let $F$ be an algebraically closed field of characteristic not equal to $2$ or $3$.

\begin{definition} \label{def:Skly3thcr} Let $a,b,c\in F^\times$ be such that 
$$
(3abc)^3\ne (a^3+b^3+c^3)^3.
$$
The {\it three-dimensional Sklyanin algebra}, denoted by $S(a,b,c)$ is generated over $F$ by $x$, $y$, $z$ with defining relations
\begin{equation*} \label{eq:S(abc)}
\begin{array}{c}ayz+ bzy+ cx^2 ~=~ azx+ bxz+cy^2 ~=~ axy+ byx+ cz^2=0. \end{array}
\end{equation*}
\end{definition} 
It is known  that $S(a,b,c)$ is Koszul with Hilbert series 
$(1-t)^{-3}$ (see \cite[Theorems~6.6(ii) and~6.8(i)]{ATV} and a result of J. Zhang \cite[Theorem~5.11]{Smi}), so that $S(a,b,c)$ is a flat deformation of the algebra of polynomials in three variables (see, e.g. \cite[Theorem~1.1]{TV}). Moreover, the center of $S(a,b,c)$ 
contains an element $T$ of degree $3$, and $S(a,b,c)/(T)=B(E,\sigma,{\mathcal{L}})$, where $E$ is the elliptic curve in $\Bbb P^2$ given by the equation
$$
(a^3+b^3+c^3)xyz=abc(x^3+y^3+z^3)
$$
with $\sigma$ given by translation by the point $(a:b:c)\in E$,
and ${\mathcal{L}}$ is a line bundle of degree $3$ on $E$.

\begin{theorem} \label{mainSkly} Let $S(a,b,c)$ be a 3-dimensional Sklyanin algebra over $k$ and let $H$ be a semisimple Hopf algebra over $k$ of dimension $d$. If  the order of $\sigma \in \Aut(E)$ is coprime to $d!$ or infinite,  
 then any $H$-action on $S(a,b,c)$ factors through a group action.
\end{theorem}

\begin{proof} It is known from the theory of Sklyanin algebras that 
if $\sigma$ has order $N$ then $S(a,b,c)$ is PI with PI degree $N$ (see \cite[part 5 of Theorem on page 7]{AST}). 
Therefore, Theorem \ref{mainSkly} is proved similarly to Theorem~\ref{main3}, 
using Corollary~\ref{ntcor} for elliptic curves, as in Theorem~\ref{main4}. 
\end{proof}

\begin{remark} 
The semisimplicity condition on $H$ in Theorem \ref{main5} 
and the infinite order condition in Theorem \ref{ellcurve} cannot be dropped, 
as there exists a Sweedler Hopf algebra action on $B:=B(X,\sigma,{\mathcal L})$ 
if $\sigma$ has odd order~$N$.  Namely, we take a sufficiently large odd number $m$ such that the line bundle ${\mathcal{L}}^{\otimes m}$ is very ample (it exists since $\mathcal{L}$ is ample). Now $B[mN]\neq 0$ and there exists an eigenvector $f$ of $\sigma$ in $B[mN]$. We then take $z=f^N$, a nonzero central element of odd degree $mN^2$, so that
a desired action is given by Example~\ref{Swee}. 

Also, the semisimplicity assumption in Theorem \ref{mainSkly} 
cannot be \linebreak dropped, as there exists a Sweedler Hopf algebra action on 
$S(a,b,c)$ for any $a,b,c$, given by Example \ref{Swee} 
where we use the central element $T$ in place of the element $z$. 
\end{remark}

\section{Appendix} 

The goal of this Appendix is to provide number-theoretic results needed in Section 4. 
We start with quoting a result from \cite{Per} (in which we take $F$ to be the number field $K$ itself).

\begin{theorem}\label{ntresult} \cite[Theorem~7]{Per} 
Let $G$ be the product of an abelian variety and a torus defined over a number field $K$. 
Let $g\in G(K)$ be a $K$-rational point on $G$ such that the Zariski closure $G_g$
of the subgroup $\langle g\rangle \subset G(K)$ generated by $g$ is connected. Fix a positive integer $r$.
Then there exists a positive Dirichlet density of primes
${\mathfrak{p}}$ of $K$ such that the order of $g$ modulo ${\mathfrak{p}}$
is coprime to $r$. \qed
\end{theorem} 

\begin{corollary}\label{ntcor}  Let $K,G$ be as in Theorem \ref{ntresult}, let 
$g\in G(K)$, and let $\ell:=|G_g/G_g^0|$, where $G_g^0$ is the connected component of the identity in $G_g$
(i.e., $G_g/G_g^0=\Bbb Z/\ell\Bbb Z$). 
Fix a positive integer $r$ coprime to $\ell$. Then, there exists a positive Dirichlet density of primes
${\mathfrak{p}}$ of $K$ such that the order of $g$ modulo ${\mathfrak{p}}$
is coprime to $r$.
\end{corollary} 

The corollary above is used in the proof of Theorem~\ref{main3} where $d!$ is~$r$ and $N_{\mathfrak p}$ is the order of $g$ modulo $\mathfrak{p}$.

\begin{proof} The order of $g$ in $G_g/G_g^0$ is $\ell$,  
hence $G_{g^{\ell}}=G_g^0$. Now the statement follows by applying Theorem 
\ref{ntresult} to $g^{\ell}$.  
\end{proof} 

\begin{example}\label{exaaa} Let $G$ be a split $m$-dimensional torus, and take an element $g:=(q_1,\dots,q_m)$ $\in G$. We have the following statements.

\begin{enumerate}
\item The group $G_g$ is connected if and only if the group $\Gamma$ generated 
by $q_1,\dots,q_m$ in $K^\times$ is free, i.e., does not contain non-trivial roots of unity. 
Indeed both conditions are equivalent to the condition that 
any character $\chi$ of $G$ which maps $g$ to an $\ell$-th root of unity
satisfies $\chi(g)=1$. 
\medskip

\item More generally, $|G_g/G_g^0|=\ell$ 
if and only if the group of roots of unity generated by $\chi(g)$, where 
$\chi$ runs through characters of $G$ such that $\chi(g)$ is a root of unity, is 
$\mu_{\ell}$. In other words, $\ell$ is the order of the torsion subgroup in 
$\Bbb Z^m/g^\perp$, where $g^\perp$ is the subgroup of characters  
$\chi$ such that $\chi(g)=1$. In particular, $\ell$ depends only on the multiplicative relations satisfied by $q_{ij}$. 
\medskip

\item If $\dim G=1$ (i.e., $G= \Bbb G_m$ or an elliptic curve), 
then $G_g$ is connected if and only if $g$ has infinite order or $g=1$. 
More generally, $|G_g/G_g^0|=\ell>1$ if and only if $g$ has order  $\ell$.  
\end{enumerate}
\end{example}

\section*{Acknowledgments}
We thank Bjorn Poonen for many useful discussions and for the number-theoretic reference \cite{Per}, which is crucial for our arguments.
We are also grateful to R. Bezrukavnikov, I. Losev, and H. Nakajima 
 for useful discussions and explanations. We thank the referee for many useful comments  that improved greatly the quality of this manuscript. 
The authors were supported by the National Science Foundation: NSF-grants DMS-1502244 and DMS-1550306.

\end{document}